\documentclass[12pt]{amsart}
\usepackage{amsfonts}
\usepackage{fullpage, amsmath, amsthm,amsfonts,amssymb,stmaryrd, mathrsfs}
\usepackage{hyperref}

\usepackage[pdftex]{color}
\usepackage[all]{xy}

\numberwithin{equation}{subsection}
\newtheorem{theorem}[subsection]{Theorem}
\newtheorem{lemma}[subsection]{Lemma}

\theoremstyle{definition}

\newtheorem{definition}[subsection]{Definition}

\newtheorem{remark}[subsection]{Remark}

\newtheorem{notation}[subsection]{Notation}

\newcommand{\ndet}{\mathrm{NormDet}}

\newcommand{\Z}{{\mathbb Z}}
\newcommand{\R}{{\mathbb R}}

\begin{document}
	
\title{An Efficient Version of the Bombieri-Vaaler Lemma
  % on Linear Diophantine Systems
}
	\author{Jun Zhang}
	\address{Jun Zhang, School of Mathematical Sciences, Capital Normal University, Beijing 100048, P.R. China.}
	\email{junz@cnu.edu.cn}	
	\author{Qi Cheng}
	\address{Qi Cheng, School of Computer Science, University of Oklahoma, Norman, OK 73019, USA.}
	\email{qcheng@ou.edu}
	\date{\today}
	\begin{abstract}
    In their celebrated paper \textit{On Siegel's Lemma}, Bombieri and Vaaler
    found an upper bound on the height of integer solutions of
    systems of linear Diophantine equations. Calculating the bound
    directly, however, requires exponential time.
    In this paper, we present the bound in a different form
    that can be computed in polynomial time.
    We also give an elementary (and arguably simpler) proof for the bound.
    % on in the paper    by . We study the relationship of the determinants of an
    % integer lattice and its kernel lattice. We give the exact formula for the
    % determinant of the integer solution lattice. We emphasize that our result is
    % more natural as it only involves lattices generated by rows and columns of
    % the coefficients matrix, and more computation-friendly as there is
    % polynomial time algorithms to finding Hermite Normal Form for a given
    % matrix, while in their work
  \end{abstract}
 %	\subjclass
%	\keywords{Eigencurves, slope of $U_p$-operators, quaternionic automorphic forms, overconvergent modular forms, Gouv\^ea--Mazur Conjecture, Gouv\^ea's conjecture on slopes}
	\maketitle
	\section{Introduction}
	Solving Diophantine equations is at the heart of mathematics. It has applications
  in many branches of mathematics and computer science.
  The basic problem in the Diophantine
  theory is to solve a system of linear Diophantine equations.
	
	\begin{definition}
	   The system of linear Diophantine equations with integer coefficient matrix $A=(a_{i,j})\in \mathbb{Z}^{k\times n}\,(k<n)$ is
	   \begin{equation}\label{diophantine}
	     AX^T=0
	   \end{equation}
	   where $X^T$ is the transpose of the vector $X=(x_1,\cdots,x_n)$ of
     variables. The goal is to find solutions in the ring of integers.
	\end{definition}
	
	The system obviously has the zero solution $(0,0,\cdots, 0)$.
  The fundamental problem in the Diophantine approximation theory~\cite{cassels}
  asks how small a nonzero integer solution can be. It is also related to the shortest vector
  problem (SVP) in computational lattice theory~\cite{micciancio-goldwasser}.
	Siegel showed the following bound.
	\begin{theorem}[\cite{siegel}]
		% Notations as above.  Then
    System~\ref{diophantine} has a nonzero integer solution $X=(x_1,\cdots,x_n)\in \mathbb{Z}^n$ such that
		\[
		  \max_i|x_i| \leq 1+(na)^{\frac{k}{n-k}},
		\]
    where $a$ is an upper bound for absolute coefficients $|a_{i,j}|$ in $A$.
	\end{theorem}
	
	Working on the ad\`eles, Bombieri and Vaaler applied the techniques from Geometry of Numbers,  and improved Siegel's bound.
	\begin{theorem}[\cite{bombieri-vaaler}]\label{BV1}
		% Notations as above.
    Suppose that $A$ has full row rank $k$. System~\ref{diophantine} has a
    nonzero
    integer solution $X=(x_1,\cdots,x_n)\in \mathbb{Z}^n$ such that
		\[
		\max_i|x_i| \leq \left(D^{-1}\sqrt{\det(A\cdot A^T)}\right)^{\frac{1}{n-k}},
		\]
		where $D$ is the greatest common divisor (G.C.D.) of determinants of all $k\times k$ minors of $A$.
	\end{theorem}

  They also proved a stronger version of Theorem~\ref{BV1},
  similar to Minkowski's second theorem for successive minima.
   	\begin{theorem}[\cite{bombieri-vaaler}]\label{BV2}
   	With the same hypotheses in Theorem~\ref{BV1}, System~\ref{diophantine} has $n-k$ linear independent integer solutions $X_j=(x_{j,1},\cdots,x_{j,n})\in \mathbb{Z}^n,\,j=1,\cdots,n-k$ such that
   	\[
   	\prod_{1\leq j \leq n-k} \max_i|x_{j,i}| \leq D^{-1}\sqrt{\det(A\cdot A^T)},
   	\]
   	where $D$ is the G.C.D. of determinants of all $k\times k$ minors of $A$.
   \end{theorem}
   To compute $ D $ directly, one has to find
   determinants of all
   $\binom{n}{k}$ square matrices which would take exponential time,
   e.g., when
    $k = (1/2 +\epsilon) n$.
	In this paper, we will prove an efficient verison of Bombieri and Vaaler's
  bounds. We  work on lattices generated by $A$ and the corresponding kernel
  lattice, and resort only to the basic facts in linear algebra.
%	Without loss of generality, in the following context, the coefficients matrix $A \in \mathbb{Z}^{k\times n}$ is always assumed to be of full row rank $k$.
	
	\begin{definition}
		An (integer) lattice $\Lambda$ is an additive subgroup of $\mathbb{Z}^n$. A family of vectors $v_i\in\Lambda,\,i=1,2,\cdots,k,$ is called a basis of the lattice $\Lambda$ if every vector $v\in\Lambda$ has the unique representation
		\[
		  v=\sum_{i=1}^k c_iv_i,\quad\mbox{for some}\quad c_i\in\mathbb{Z}.
		\]
		Here, we call $k$ and $n$ the rank and the dimension of the lattice
    respectively. We also say that the lattice $\Lambda$ is generated
    by $v_i\in\Lambda,\,i=1,2,\cdots,k$, and call the matrix
    $V=(v_1^T,\cdots,v_k^T)^T$
    with rows $v_i$
    a basis matrix for $\Lambda$.
	\end{definition}

	The determinant of a lattice plays an important role in the lattice theory.
  It is an invariant of a lattice independently of
  the choices of basis matrices.
	\begin{definition}
		Let $\Lambda$ be a lattice with a basis matrix $A$. The determinant of the lattice $\Lambda$ is defined to be
		\[
		  \det(\Lambda)=\sqrt{\det(A\cdot A^T)}.
		\]
		It is also called the volume of the lattice.
	\end{definition}
	
  \section{Our results}
	% The determinant of a lattice  The definition works perfect in the dual
  % lattice theory.  But we will see that it does not hold
  % anymore for the kernel lattice scenario.
  We now define the kernel lattice. One can contrast it with the better-known
  dual lattice.
	
	\begin{definition}
		The integer solutions of System~\ref{diophantine} forms an additive group.
    It is called the kernel lattice of $ \Lambda = L_R(A)$,
    and it is denoted by $\Lambda^0 $.
	\end{definition}
	
  It is well known that the determinant of the dual lattice equals to the inverse of the
  determinant of the original lattice.
  % , and a lattice and  its dual lattice
  % are dual to each other. However,
  Furthermore if we take the dual of a dual lattice, we will
  get back to the original lattice. However, it is not
  true for kernel lattices. In the other words, $\Lambda $ may not be equal
  to $( \Lambda^0 )^0$. For example, consider the lattice in $\R^2$ generated
  by the vector $(2,2)$. Its kernel lattice is generated by the vector
  $(1, -1)$, whose kernel lattice is generated by $(1,1)$. Nevertheless
  it is always true that for any integer lattice $\Lambda$,
  \[ \Lambda^0 = (( \Lambda^0 )^0)^0. \]
 First we introduce some notations:
	
	\begin{notation}
		Let $A \in \mathbb{Z}^{k\times n}$ be an integer matrix.
    Denote the lattices generated by rows and columns of $A$ by $L_R(A)$ and $L_C(A)$,
		respectively.
	\end{notation}
	
So the basis matrix of the lattice $ L_R(A)$ is $A$
    if rows of $A$ are linear independent over $\R$.
 To study the determinant of the kernel
  lattice, we define
	\begin{definition}
		Let $\Lambda$ be the lattice with a basis matrix $A$.
    % and $\Lambda^0$ be its kernel lattice.
    % The determinant of the kernel lattice of $\Lambda^0$ is called
    The normalized determinant of $\Lambda$, denoted by $\ndet(\Lambda)$,
    is defined to be
	    	\begin{equation}   %\label{key}
	    	\ndet(\Lambda)=\frac{\det(L_R(A))}{\det(L_C(A))}.%\frac{\sqrt{\det(A\cdot A^T)}}{\det(L_C(A))}.
	    	\end{equation}
	\end{definition}
	
	It is easy to see that the definition of normalized determination is
  independent of the choices of basis matrix $A$.
  Furthermore, the definition can be
  extended to $A\in \mathbb{Q}^{k\times n}$ with coefficients in rational
  numbers $\mathbb{Q}$.
	The main contribution of this paper is to relate the determinant of
  kernel lattice with the
  normalized determinant.
	\begin{theorem}\label{detofkernel}
		Let $\Lambda$ be an integer lattice and $\Lambda^0$ be its kernel lattice. Then
		\begin{equation}\label{key}
		\det(\Lambda^0)=\ndet(\Lambda).
		\end{equation}
	\end{theorem}

     So the normalized determinant can be viewed as the determinant of
     its kernel lattice.
  The following theorem shows that the normalized determinant is invariant under
  the kernel operation.
  % In the other words,
		 % any normalized integer lattice and its kernel lattice have the same volume.
	 \begin{theorem}\label{main_dual}
	 	Let $A\in \mathbb{Z}^{k\times n}$ and $B\in \mathbb{Z}^{(n-k)\times n}$ be two integer matrices of full row rank such that $A\cdot B^T=0$. Then we have
	 	\[
	 	   \ndet(L_R(A))=\ndet(L_R(B)).
	 	\]
	 \end{theorem}

	   % As an immediate consequence, we obtain the determinant of integer solution
     % lattice of linear Diophantine systems.
     Now applying Vaaler's cube slicing
     inequality~\cite{vaaler79} and Minkowski's theorems~\cite{davenport39} on
     the integer solution lattice of Diophantine System~\ref{diophantine},
	
	\begin{theorem}
		Suppose that $A\in \mathbb{Z}^{k\times n}$ has full row rank $k$ and
    $\Lambda$ is the lattice generated by rows of $A$. System~\ref{diophantine}
    has a nonzero integer solution $X=(x_1,\cdots,x_n)\in \mathbb{Z}^n$ such that
		\[
      \max_i|x_i| \leq \ndet(\Lambda)^{\frac{1}{n-k}}.
      % =\left(\frac{\sqrt{\det(A\cdot A^T)}}{\det(H_A)}\right)^{\frac{1}{n-k}},
		\]
	%	where $(H_A|0_{k\times (n-k)})$ is the column Hermite Normal Form of $A$.
	Moreover, System~\ref{diophantine} has $n-k$ linear independent integer solutions $X_j=(x_{j,1},\cdots,x_{j,n})\in \mathbb{Z}^n,\,j=1,\cdots,n-k$ such that
		\[
      \prod_j\max_i|x_{j,i}| \leq \ndet(\Lambda).
      % =\frac{\sqrt{\det(A\cdot A^T)}}{\det(H_A)}.
		\]
	\end{theorem}
	
	Finally, to recover Bombieri and Vaaler's theorems~\ref{BV1} and~\ref{BV2}, we only need to show that
	\begin{theorem}\label{HNFandD}
    Suppose that $A\in \mathbb{Z}^{k\times n}$ has full row rank $k$.
    % and columns of $A$ has the Hermite Normal Form $(H_A|0_{k\times (n-k)})$ (see
    % Definition~\ref{HNFdef}), where $(0_{k\times (n-k)})$ denotes the $k\times
    % (n-k)$ zero matrix.
    Denote by $D$ the G.C.D. of determinants of all $k\times
    k$ minors of $A$. Then we have
	 $$\det(L_C( A ))=D.$$
	\end{theorem}

	\begin{remark}
    From the computational aspect, there is a polynomial time algorithm to
    compute $\det(L_C (A))$, while computing $D$ takes exponential time as one needs
    to compute determinants of $\binom{n}{k}$ $k\times k$-minors of $A$ and then
    take the G.C.D.. Using G.C.D. of determinants of all full-rank minors of $A$
    to bound the solution of \ref{diophantine}
    was rediscovered in~\cite{brudern-dietmann}, which extends the special case
    where the G.C.D is one in~\cite{health-brown}.
    \end{remark}
 	
	\section{Proofs}
	In this section, we give proofs of Theorems~\ref{main_dual}
  and~\ref{HNFandD}. We first review the definition of the Hermite Normal Form
  and prove a simple technical lemma.
  	\begin{definition}\label{HNFdef}
		For any matrix $M\in \mathbb{Z}^{m\times n}$, there is a square unimodular matrix $U\in \mathbb{Z}^{n\times n}$ such that $H=M\cdot U$ is of the form:
	\begin{itemize}
		\item $H$ is lower triangle, i.e., $h_{i,j}=0$ for all $i<j$.
		\item The leading coefficient $h_{i,i}$ of a nonzero column is positive and $h_{i,i}>h_{j,i}$ for all $j>i$.
		\item Zero columns are arranged on the right.
	\end{itemize}
    We call $H$ is the column Hermite Normal Form (HNF) of $M$.
	\end{definition}

	\begin{lemma}\label{normdet}
		Let $\Lambda$ be the lattice with a basis matrix $A$ and $\Lambda^0$ be its
    kernel lattice. Suppose that columns of $A$ has the HNF $(H_A|0_{k\times (n-k)})$, i.e.,
		\[
		   A\cdot U=(H_A|0_{k\times (n-k)})
		\]
	    for some unimodular matrix $U=\left(\begin{array}{cc}
	U_{1,1}&U_{2,1}\\
	U_{1,2}&U_{2,2}
	\end{array}\right)\in \mathbb{Z}^{n\times n}$, where $(0_{k\times (n-k)})$ denotes the $k\times (n-k)$ zero matrix
and blocks $U_{1,1}\in\mathbb{Z}^{k\times k},\,U_{1,2}\in\mathbb{Z}^{(n-k)\times k},\,U_{2,1}\in\mathbb{Z}^{k\times (n-k)},\,U_{2,2}\in\mathbb{Z}^{(n-k)\times (n-k)}$, respectively. Then we have
	    \begin{enumerate}
	    	\item The determinant of column lattice $L_C(A)$ is
	    	$$
	    	\det(L_C(A))=\det(H_A).
	    	$$
	    	Moreover, it is invariant under multiplication by an invertible integer matrix on the left of $A$. So, it is independent of the choices of basis matrices for $\Lambda$.
      \item The matrix $H_A^{-1}\cdot A$ is a basis matrix of the kernel
    lattice of $\Lambda^0$, and
	    	\begin{equation*} %\label{key}
          \ndet(\Lambda)= \det((\Lambda^0)^0).
	    	\end{equation*}
      \item Rows of $(U^T_{2,1},U^T_{2,2})$ form a basis for the lattice $\Lambda^0$.
	    \end{enumerate}
	\end{lemma}
	
\begin{remark}
  We call $(\Lambda^0)^0$ the normalized lattice of $\Lambda$, whose
  basis matrix can be calculated as $H_A^{-1}A$.
 % From the geometry aspect, the expression $\sqrt{\det(H_A^{-1}A\cdot
        % (H_A^{-1}\cdot A)^T)}$ is a better choice for the determinant of $A$
      % in our context, as it is the determinant of
      % the normalized lattice generated by rows of $H_A^{-1}A$ (as the kernel
      % lattice of its kernel lattice).
	   % \begin{remark}
	   	% Here, we see that under kernel operation, normalized determinant should be the right definition, instead of determinant.
	   % \end{remark}
	\end{remark}

	\begin{proof}%[Proof of Lemma~\ref{normdet}]
		(1) From the definition of HNF, it is clear that columns of $H_A$ form a basis for the column lattice $L_C(A)$. So it follows the equality $\det(L_C(A))=\det(H_A)$. For the later statement, one can deduce it from Theorem~\ref{HNFandD}.
		% let $A'=U_1\cdot A$ be another basis matrix for $\Lambda$, where $U_1\in
    % \mathbb{Z}^{k\times k}$ is an invertible matrix. Suppose that $A'$ has the HNF
		%\[
		%A'\cdot U'=(H_{A'}|0_{k\times (n-k)})
		%\]
		%for some unimodular matrix $U'\in \mathbb{Z}^{n\times n}$.
		
    (2)
    % Indeed, we will show  Then lemma~\ref{normdet} follows directly from the
    % definition.
	One proof uses dual lattice, which explains where $H_A^{-1}$ comes from. Since rows of $A$ form an $\mathbb{R}$-basis for the kernel space of $B$, the kernel lattice of $B$ is
	\[
	\mathrm{Span}(A)\cap \mathbb{Z}^n,
	\]
	where $\mathrm{Span}(A)$ denotes the $\mathbb{R}$-vector space generated by rows of $A$. To determine the kernel lattice of $B$, it only needs to decide which vector $Y=(y_1,\cdots,y_k)\in \mathbb{R}^k$ satisfies
	\[
	(y_1,\cdots,y_k)\cdot A\in \mathbb{Z}^n.
	\]
	The set of such vectors $Y$'s forms the dual lattice of column lattice $L_C(A)$.
  And we know that columns in $H_A$ is a basis for $L_C(A)$. So by the formula
  for basis of dual lattice~\cite{micciancio-goldwasser}, $H_A^{-1}$ is a basis
  matrix for the dual lattice. Hence, $H_A^{-1}\cdot A$ is a basis matrix
  of the kernel lattice of $\Lambda^0$, and
  \begin{align*}
	 & \det((\Lambda^0)^0) = \det(L_R( H_A^{-1}\cdot A))
  = \sqrt{\det(H_A^{-1}A\cdot (H_A^{-1}\cdot A)^T)}\\
  =& \sqrt{\det(H_A^{-1}A\cdot A^T \cdot (H_A^{-1})^T)}
  =\frac{\sqrt{\det(A\cdot A^T)}}{\det(H_A)}.
  \end{align*}

	Another proof is more direct. Since $ A\cdot U=(H_A|0_{k\times (n-k)})$, we have
	\[
	H_A^{-1}\cdot A=(I_k|0_{k\times (n-k)})\cdot U^{-1}\in
	\mathbb{Z}^{k\times n},
	\]
	where $I_k$ is the $k\times k$ identity matrix. As $U^{-1}$ is unimodular,
  rows of $H_A^{-1}\cdot A$ form a subset of a  basis for lattice $\mathbb{Z}^n$,
  which is called primitive in
  \cite{cassels}.
  % $H_A^{-1}\cdot A$ itself must be tight which means
  The lattice generated by rows of $H_A^{-1}\cdot A$ can not be a
  proper sublattice of another (integer) lattice in the kernel space of $B$. So
  rows of $H_A^{-1}\cdot A$ form a basis of the kernel lattice of $B$.
	
 (3) Diophantine System~\ref{diophantine} is equivalent to
	\[
	(A\cdot U)\cdot (U^{-1}\cdot X^T)=0,
\]
or
\[ (H_A|0_{k\times (n-k)})\cdot(U^{-1}\cdot X^T)=0.
	\]
	Therefore, we get a $\mathbb{R}$-basis matrix for the solution space of $A$
	\[
	C=(0_{(n-k)\times k}|I_{n-k})\cdot U^T=(U^T_{2,1},U^T_{2,2}),
	\]
	where $I_{n-k}$ denotes the $(n-k)\times (n-k)$ identity matrix. Using the same argument in the proof of
the statement (2), one can prove that $$C=(U^T_{2,1},U^T_{2,2})$$ indeed form a basis matrix for the lattice $\Lambda^0$.
    \end{proof}

	Now, we prove Theorems~\ref{detofkernel}, ~\ref{main_dual} and~\ref{HNFandD}.

\begin{proof}[Proof of Theorem~\ref{detofkernel}]
		Without loss of generality, we may assume the first $k$ columns $A_1$ of
    $A=(A_1,A_2)$ are linearly independent. Suppose that $(H_A|0_{k\times (n-k)})$ is the column HNF of $A$ and $U=\left(\begin{array}{cc}
	U_{1,1}&U_{2,1}\\
	U_{1,2}&U_{2,2}
	\end{array}\right)$ is the unimodular transformation matrix. That is,
	\[
	\begin{cases}
	A_1\cdot U_{1,1}+A_2\cdot U_{1,2}=H_A\\
	A_1\cdot U_{2,1}+A_2\cdot U_{2,2}=0_{k\times (n-k)}.
	\end{cases}
	\]
	By Lemma~\ref{normdet}(3) we get a basis matrix for the kernel lattice $\Lambda^0$
	\[
	C=(U^T_{2,1},U^T_{2,2}).
	\]
	%That is, $C$ is a basis matrix for the kernel lattice of the kernel lattice of $B$.
    So
	\[
	\det(\Lambda^0)^2=\det(C\cdot C^T)=\det(U^T_{2,1}\cdot U_{2,1}+U^T_{2,2}\cdot U_{2,2}).
	\]
	From the assumption that $\det(A_1)\neq 0$, we have
	\[
	U_{2,1}=-A_1^{-1}\cdot A_2\cdot U_{2,2}.
	\]
	So
	\begin{align*}
	\det(\Lambda^0)^2&=\det(U^T_{2,2}(A_1^{-1}\cdot A_2)^T\cdot A_1^{-1}\cdot A_2\cdot U_{2,2}+U^T_{2,2}\cdot U_{2,2})\\
	&=\det(U_{2,2})^2\det(I_{n-k}+(A_1^{-1}\cdot A_2)^T\cdot A_1^{-1}\cdot A_2)\\
               &=\det(U_{2,2})^2\det(I_{k}+A_1^{-1}\cdot A_2\cdot (A_1^{-1}\cdot A_2)^T)
                 \stepcounter{equation}\tag{\theequation}\label{eq1}\\
	&=\det(U_{2,2})^2\det(I_{k}+A_1^{-1}\cdot A_2\cdot A_2^T\cdot (A_1^{T})^{-1})\\
	&=\frac{\det(U_{2,2})^2}{\det(A_1)^2}\det(A_1\cdot A_1^T+ A_2\cdot A_2^T)\\
	&=\frac{\det(U_{2,2})^2}{\det(A_1)^2}\det(A\cdot A^T)\\
	&=\frac{\det(A\cdot A^T)}{\det(H_A)^2}
                 \stepcounter{equation}\tag{\theequation}\label{eq2}\\
	&=\ndet(\Lambda)^2.\stepcounter{equation}\tag{\theequation}\label{eq3}\\
	\end{align*}
 Equality~\ref{eq1} follows from Sylvester's determinant identity (see
  Remark~\ref{exchange}). Equality~\ref{eq2} follows from taking
  determinants of both sides of the equality:
	\[
	\left(\begin{array}{cc}
	A_1&A_2\\
	0&I_{n-k}
	\end{array}\right)\left(\begin{array}{cc}
	U_{1,1}&U_{2,1}\\
	U_{1,2}&U_{2,2}
	\end{array}\right)=\left(\begin{array}{cc}
	H_A&0\\
	U_{1,2}&U_{2,2}
	\end{array}\right),
	\]
	and $U=\left(\begin{array}{cc}
	U_{1,1}&U_{2,1}\\
	U_{1,2}&U_{2,2}
	\end{array}\right)$ is unimodular. And Equality~\ref{eq3} follows from the definition and Lemma~\ref{normdet}(1).
\end{proof}

\begin{remark}
  \label{exchange}
  The Sylvester's determinant identity states that
		for any matrices $A\in \mathbb{R}^{s\times t}$ and $B\in \mathbb{R}^{t\times s}$
		\[
		\det(I_s+A\cdot B)=\det(I_t+B\cdot A).
		\]
    The proof can be found in many places, e.g. \cite{AAM96}.
\end{remark}

\begin{proof}[Proof of Theorem~\ref{main_dual}]
 Denote $\Lambda=L_R(A)$. By Theorem~\ref{detofkernel},
 \[
   \ndet(L_R(B))=\det((L_R(B))^0)=\det((\Lambda^0)^0).
 \]

 On the other hand, we have already showed that $\ndet(\Lambda)$ equals the determinant of the kernel lattice
 of $\Lambda^0$ by statements (1) and (2) of Lemma~\ref{normdet}, i.e.,
 \[
  \ndet(\Lambda)=\det((\Lambda^0)^0).
 \]
 So
 \[\ndet(L_R(A))=\ndet(L_R(B)).
 \]
\end{proof}

   \begin{proof}[Proof of Theorem~\ref{HNFandD}]
     For any integer matrix $M \in \Z^{k\times n}$, let $D(M)$ denote the
     G.C.D of the determinants of all $k \times k$ minors.
     We claim that if an integer matrix $M'$ is the result of
     an integer elementary column operation from $M$, then $ D(M) | D(M')$.
     It is obviously true if the operation is a column switching, or a column
     multiplication (by an integer). The remaining case is a column addition.
     In this case, the determinant of a $k\times k$ minor of $M'$
     is an integral linear combination of determinants of some
     $k \times k$ minors of $M$. So again we have $ D(M) | D(M')$.

     For an integer matrix $A \in \Z^{k\times n}$, after a sequence
     of invertible integer column operations, we can obtain of a matrix of
     form $(H_A|0_{k\times (n-k)}).$ We have $D ( (H_A|0_{k\times (n-k)}) ) =
     \det(L_C (A))$,
     and
     \[ D(A) |  D ( (H_A|0_{k\times (n-k)}) ) \mbox{\ and\ }
     D ( (H_A|0_{k\times (n-k)}) ) | D(A).\]
     So $ D(A) = \det(L_C (A))$.
   %   On one hand, columns of $H_A$ is a basis of the lattice generated by
   %   columns of $A$, so $\det(H_A)$ divides the determinant of any $k\times k$
   %   minors of $A$, and hence $\det(H_A)|D$. On the other hand, notice that
   %   invertible column transformations do not change $D$, and $H_A$ is obtained
   %   by a series of invertible column transformations, so $D|\det(H_A)$. Since
   %   both are positive numbers, we have
   %   $\det(H_A)=D$.
   \end{proof}

	\section*{Acknowledgment}
	 Jun Zhang is supported by the National Natural Science Foundation of China No. 11601350, by Scientific Research Project of Beijing Municipal Education Commission under Grant No. KM201710028001, by Beijing outstanding talent training program No.2014000020124G140 and by China Scholarship Council. He would like to thank University of Oklahoma for the hospitality during his visit. Qi Cheng is supported by China 973 program No. 2013CB834201 and by US NSF No. CCF-1409294.

\end{document}